\documentclass[12pt]{amsart}

\usepackage{amssymb}

\usepackage{enumerate}

\usepackage{graphicx}

\makeatletter
\@namedef{subjclassname@2010}{%
  \textup{2010} Mathematics Subject Classification}
\makeatother

\newtheorem{thm}{Theorem}[section]
\newtheorem{cor}[thm]{Corollary}
\newtheorem{lem}[thm]{Lemma}

\theoremstyle{definition}
\newtheorem{defin}[thm]{Definition}
\theoremstyle{remark}
\newtheorem{rem}[thm]{Remark}
\theoremstyle{problem}

\theoremstyle{question}
\newtheorem{quest}[thm]{Question}

\numberwithin{equation}{section}

\begin{document}


\baselineskip=17pt



\title[On structure of all real valued  sequences $\cdots$]{On structure of all real valued  sequences uniformly distributed in $[-1/2, 1/2]$
from the point of view of shyness
}

\author[Gogi R. Pantsulaia]{Gogi R. Pantsulaia}
\address{I.Vekua Institute of Applied Mathematics,
Tbilisi State University, 0143 Tbilisi, Georgian Republic.\newline
Department of Mathematics, Georgian Technical University,
0175 Tbilisi, Georgian Republic}
\email{g.pantsulaia@gtu.ge}


\date{}

\begin{abstract} In the paper [Inter. J. Sci. Tech., {\bf 4(3)} (2013), 21--27], it was shown  that $\mu$-almost every element of $\mathbf{R}^{\infty}$
is uniformly distributed in $[-\frac{1}{2}, \frac{1}{2}]$, where $\mu$ denotes Yamasaki-Kharazishvili  measure in $\mathbf{R}^{\infty}$
for which $\mu([-\frac{1}{2},\frac{1}{2}]^{\infty})=1$. In the present paper the same set is studying from the point of view of shyness and it is demonstrated that it is shy  in $\mathbf{R}^{\infty}$. In Solovay model,  the set of all real valued  sequences uniformly distributed  modulo 1 in $[-\frac{1}{2}, \frac{1}{2}]$ is studied from the point of view of shyness and it is shown that  it is prevalent set in $\mathbf{R}^{\infty}$.
\end{abstract}



\maketitle

\section{Introduction}

It is well known that now days  the  theory of  uniform  distribution  has  many  interesting  applications  in various  branches of mathematics, such  are numbers theory, probability
theory, mathematical  statistics, functional analysis, topological algebra, and so on. Therefore, research of internal structure of all uniformly distributed sequences doesn't lose the relevance to this day. For example, in \cite{Pan2013} has been  studied the  set $D$ of all
real valued sequences uniformly distributed in $[-\frac{1}{2},\frac{1}{2}]$ in
terms of the Yamasaki-Kharazishvili   measure $\mu$ \cite{Yamasaki80}-\cite{Kharaz84} and has been demonstrated that $\mu$-almost every element of $\mathbf{R}^{\infty}$ is
uniformly distributed in $[-\frac{1}{2}, \frac{1}{2}]$.

The purpose of the present paper is to study structures of  $D$  and  $F$ from the point of view of shyness \cite{HSY92}, where
$F$ denote the set  of all real valued sequences uniformly distributed modulo $1$ in $[-\frac{1}{2}, \frac{1}{2}]$.

The rest of the manuscript is the following.

In Section 2 we consider some auxiliary notions and facts from  mathematical analysis
 and measure theory. In Section 3 we prove that $D$ is shy in $R^{\mathbb{N}}$.
In Section 3, we demonstrate that in Solovay Model \cite{Solovay1970}  the set $F$ is prevalent set \cite{HSY92}  in $R^{\mathbb{N}}$.

\section{Some auxiliary notions and facts from  mathematical analysis
 and measure theory}

Let consider some notions and auxiliary facts from  mathematical analysis
 and measure theory which  will be  useful  for our further investigations.
\begin{defin}
A sequence of real
numbers $(x_k)_{k \in \mathbb{N}} \in \mathbf{R}^{\infty}$ is called
uniformly distributed in $[a,b]$( abbreviated u.d. in $[a,b]$) if for each $c,d$ with $a \le c
<d \le b$  we have
\begin{align}
\lim_{n \to \infty}\frac{\#(\{ x_k : 1 \le k \le n\} \cap
[c,d])}{n}=\frac{d-c}{b-a},
\end{align}
where $\#(\cdot)$ denotes the counter measure of a set.
\end{defin}

Let $\lambda$ be the Lebesgue measure on $[0,1]$. This measure induces the product measure $\lambda_{\infty}$  in $[0,1]^{\infty}$.

\begin{lem} (\cite{KuiNie74}, Theorem 2.2, p. 183)~Let $S$ be the set of all sequences u.d. in $[0,1]$, viewed as
a subset of $\mathbf{R}^{\infty}$. Then $\lambda_{\infty}(S \cap [0,1]^{\infty}) = 1$.
\end{lem}

Let $~{V}$ be a complete metric linear space, by which we
mean a vector space (real or complex) with a complete metric for
which the operations of addition and scalar multiplication are continuous. When we
speak of a measure on $~{V}$ we will always mean a
nonnegative measure that is defined  on the Borel sets of
$~{V} $ and is not identically zero. We write $S + v$ for
the translation of a set $S \subseteq ~{V} $ by a vector $v
\in ~{V} $.

\begin{defin}[ \cite{HSY92}, Definition 1, p. 221]  A measure $\mu$ is said to be transverse to a Borel
 set $S \subset ~{V} $ if the
following two conditions hold:
\begin{enumerate}[\upshape (i)]
\item There exists a compact set $U \subset ~{V} $ for which
$0 < \mu(U) <1$;
\item $\mu(S + v) = 0$ for every $v \in ~{V} $.
\end{enumerate}
\end{defin}

\begin{defin}[ \cite{HSY92}, Definition 2, p. 222]
A Borel set $S \subset ~{V} $ is called shy  if there exists
a measure transverse to $S$. More generally, a subset of
$~{V} $ is called shy if it is contained in a Borel shy set.
The complement of a shy set is called a prevalent set.
\end{defin}

\begin{defin}( \cite{Pan08-1}, Definition 2.4,  p.237)~A Borel measure $\mu$ in $V$ is called a
generator of shy sets in $V,$ if
\begin{align}
(\forall X)(\overline{\mu}(X)
= 0  \rightarrow  X \in S(V )),
\end{align} where $\overline{\mu}$ denotes a
usual completion of the Borel measure $\mu$, where  $S(V)$ denotes the $\sigma$-ideal of all shy sets in $V$.
\end{defin}

\begin{lem}( \cite{Pan08-1}, Theorem 2.4, p. 241) {\it  Every
quasi-finite\footnote{A measure $\mu$ is called quasi-finite if there exists a $\mu$-measurable set $A$ for which $0<\mu(A)<+\infty$.} translation-quasi-invariant \footnote{A Borel measure $\mu$
defined in a Polish topological vector space $V$ is called translation-quasi-invariant if for each $\mu$-measurable set $A$ and any $h \in V$, the following  conditions $\mu(A)=$ and $\mu(A+h)=0$ are equivalent. } Borel measure $\mu$
defined in a Polish topological vector space $V$ is a generator of shy sets.}
\end{lem}

The key ingredient for our investigation will be  well known lemma from the probability theory.

\begin{lem}(Borel-Canteli lemma)(\cite{Shiryaev2012}, p.271)  Let $(\Omega,{\bf F},P)$ be a probability space. Let $(E_n)_{n \in N}$ be a sequence of events such that
\begin{align}\sum_{n=1}^\infty P(E_n)<\infty.
\end{align}
Then the probability that infinitely many of them occur is 0, that is,
\begin{align}
P\left(\limsup_{n\to\infty} E_n\right) = 0.
\end{align}
Here, "$\limsup_{n\to\infty}$" denotes limit supremum of the sequence of events $(E_n)_{n \in N}$ which is defined by
\begin{align}
\limsup_{n\to\infty} E_n = \bigcap_{n=1}^{\infty} \bigcup_{k=n}^{\infty} E_k.
\end{align}
\end{lem}

Below we present a certain example of translation-invariant Borel measure in $R^N$ in the Solovay Model (SM) \cite{Solovay1970}  which is
the following system of axioms:
\begin{align}
(ZF)\&(DC)\&(\mbox{every subset of}~{\bf{R}}~\mbox{is measurable
in the Lebesgue sense}),
\end{align}
where $(ZF)$ denotes the  Zermelo-Fraenkel set theory and $(DC)$
denotes the axiom of Dependent Choices.

\begin{lem}( \cite{Pan04-2},  Corollary 1, p. 64 ) (SM)  Let $\Bbb{J}$ be any non-empty subset of
the set all natural numbers $\Bbb{N}$.  Let, for $k \in
\Bbb{J}$,~$ S_k$ be the unit circle in the Euclidean plane
${\bf{R}}^2$. We may identify
 the circle $S_k$ with a compact group of all rotations of ${\bf{R}}^2$~ about its origin. Let $\lambda_{\Bbb{J}}$ be the probability Haar measure  defined on the compact group
$\prod \limits_{k \in {\Bbb{J}}}S_k$. Then the
completion $\overline{\lambda_{\Bbb{J}}}$ of $\lambda_{\Bbb{J}}$
is defined on the power set of $\prod \limits_{k \in
\Bbb{J}}S_k.$
\end{lem}

 For $k \in \Bbb{N}$, define the function $f_k$
by $ f_k(x)=\exp\{ 2 \pi x i\}$ for every $x \in {\bf{R}}$.

 For $E \subset {\bf{R}}^{\Bbb{N}}$ and $g \in \prod \limits
_{k \in \Bbb{N}}S_k$, put
 \begin{align}
 f_{E}(g)= \left\{
\begin{array}{cl}
\mbox{card}( (\prod\limits_{k \in \Bbb{N}}f_k)^{-1}(g) \cap E), & \mbox{if this is finite;}\\
+ \infty , & \mbox{in all other cases.}
\end{array} \right.
\end{align}

Define the functional $\mu_{\Bbb{N}}$ by
\begin{align}
(\forall E)(E \subset {\bf{R}}^{\Bbb{N}} \rightarrow
\mu_{\Bbb{N}}(E)= \int \limits_{\prod \limits_{k \in
\Bbb{N}}S_k}f_E(g)d \overline{\lambda_{\Bbb{N}}}(g)).
\end{align}

\begin{lem} ( \cite{Pan04-2},  Lemma 3, p. 65 )(SM) $\mu_{\Bbb{N}}$ is a translation-invariant measure defined on the
powerset ${\bf{R}}^{\Bbb{N}}$ such that $\mu_{\Bbb{N}}([0;1]^{\Bbb{N}})=1$.
\end{lem}

\section{On structure of real valued sequences  uniformly distributed in $[-\frac{1}{2},\frac{1}{2}]$
from the point of view of shyness}

 Let $(x_k)_{k \in \mathbb{N}} \in D$. We set $J=\{k: x_k \notin [-\frac{1}{2},\frac{1}{2}]\}$. Note that for $J$ we must have
\begin{align}\lim_{n \to \infty} \frac{\# (J \cap [0,n])}{n}=0.
\end{align}
In other words, the density of $J$ must be equal to zero. Indeed, we have
\begin{align*}1=\lim_{n \to \infty} \frac{\# (\{x_1, \cdots, x_n\}\cap [-1/ 2,1/2])}{n}=
\end{align*}
\begin{align*}\lim_{n \to \infty} \frac{\# (\{x_k : k \in J\cap [0,n] \} \cup \{x_k : k \in (\mathbb{N}\setminus J)\cap [0,n]\} \cap [-1/ 2,1/ 2])}{n}=
\end{align*}
\begin{align*}\lim_{n \to \infty} \frac{\# (\{x_k : k \in J\cap [0,n] \} \cap [-1/ 2,1/ 2])}{n}+
\end{align*}
\begin{align*}
\lim_{n \to \infty} \frac{\# (\{x_k : k \in (\mathbb{N}\setminus J)\cap [0,n]\} \cap [-1/ 2,1/ 2])}{n}=
\end{align*}
\begin{align}\lim_{n \to \infty} \frac{\#((\mathbb{N} \setminus J) \cap [0,n])}{n}.
\end{align}
The latter relation means that the density of $\mathbb{N} \setminus J$ is equal to $1$ which implies that the density of $J$ is equal to zero.

Let $\mathbb{J}$ be the class of all subsets of $\mathbb{N}$ whose densities are equal to zero. Following above mentioned discussion we  conclude that \begin{align}D \subseteq \cup_{J \in \mathbb{J}}R^J \times [-1/ 2,1/ 2]^{\mathbb{N} \setminus J}.
\end{align}
It is not hard to show that $R^J \times [-1/ 2,1/ 2]^{\mathbb{N} \setminus J}$ is shy for each $J \in \mathbb{J}$. It can be showed as follows: for each $J \in \mathbb{J}$ we have that $[-1/ 2,1/ 2]^{\mathbb{N} \setminus J}$ is compact in $R^{\mathbb{N} \setminus J}$ because of infinity of the set $\mathbb{N} \setminus J$ which implies that $[-1/ 2,1/ 2]^{\mathbb{N} \setminus J}$ is shy in $R^{\mathbb{N} \setminus J}$ (see, \cite{HSY92}, Fact 8, p. 226).  Hence, a set $R^J \times [-1/ 2,1/ 2]^{\mathbb{N} \setminus J}$, as a product of two sets between of them at least one is shy, again is shy.
Note that cardinality of the class $\mathbb{J}$ is equal to $c$, where $c$ denotes the cardinality of the continuum. The latter relation follows from the fact that there is $J_0 \in \mathbb{J}$ for which $\mbox{card}(J_0)=\aleph_0$. Then each subset of  $J_0$ also belongs to the class $\mathbb{J}$ which gives a required result.

It is obvious that the class $\mathbb{J}$ admits the following representation $\mathbb{J}=\mathbb{J}_{finite} \cup \mathbb{J}_{infinite}$, where
$\mathbb{J}_{finite}$ and $\mathbb{J}_{infinite}$ denote those classes  of  elements of $\mathbb{J}$ which are finite and infinite, respectively. Note that
\begin{align*}\cup_{J \in \mathbb{J}}R^J \times [-1/ 2,1/ 2]^{\mathbb{N} \setminus J}=\cup_{J \in \mathbb{J}_{finite}}R^J \times [-1/ 2,1/ 2]^{\mathbb{N} \setminus J} \cup
\end{align*}
\begin{align}
\cup_{J \in \mathbb{J}_{infinite}}R^J \times [-1/ 2,1/ 2]^{\mathbb{N} \setminus J}
\end{align}

It can be shown that for Yamasaki-Kharazishvili measure $\mu$ we have
\begin{align}
\mu(\mathbf{R}^{\mathbb{N}} \setminus \cup_{J \in \mathbb{J}_{finite}}R^J \times [-1/ 2,1/ 2]^{\mathbb{N} \setminus J})=0.
\end{align}

There naturally arises the following  question.

\begin{quest}  Is the set $\cup_{J \in \mathbb{J}}R^J \times [-1/ 2,1/ 2]^{\mathbb{N} \setminus J}$  shy in $R^\mathbb{N}$?
\end{quest}

\begin{lem} The answer to Question 3.1 is yes.
\end{lem}

\begin{proof} Let $E_n=[-\frac{1}{2}, \frac{1}{2}]_n \times \mathbf{R}^{\mathbb{N}  \setminus \{n\}}$, i.e.

\begin{align}
E_n=\{ (x_k)_{k \in \mathbb{N}}: x_n \in [-\frac{1}{2}, \frac{1}{2}]~\& ~x_k \in \mathbf{R}~\mbox{for}~ k \in \mathbb{N} \setminus \{n\}\}
\end{align}
Note that it is sufficient to show that $\limsup_{n\to\infty} E_n$ is a Borel shy set in $\mathbf{R}^\mathbb{N}$ because
\begin{align}
\cup_{J \in \mathbb{J}_{infinite}}R^J \times [-1/ 2,1/ 2]^{\mathbb{N} \setminus J} \subseteq \limsup_{n\to\infty} E_n.
\end{align}
Indeed, if $(x_k)_{k \in \mathbb{N}} \in \cup_{J \in \mathbb{J}}R^J \times [-1/ 2,1/ 2]^{\mathbb{N} \setminus J}$ then
there will be  a null-dense  subset $J_0$ such that $(x_k)_{k \in \mathbb{N}} \in R^{J_0} \times [-1/ 2,1/ 2]^{\mathbb{N} \setminus J_0}$.  Since
the density of $J_0$ in $\mathbb{N}$  is equal to zero we deduce that  $\mathbb{N} \setminus J_0$ is infinite. Then it is obvious that
\begin{align}
R^{J_0} \times [-1/ 2,1/ 2]^{\mathbb{N} \setminus J_0} \subseteq  \limsup_{n\to\infty} E_n.
\end{align}

For $m \in R, \sigma>0$,  we put:

\begin{enumerate}[\upshape (i)]
\item $\xi_{(m,\sigma)}$ is a Gaussian random variable on $R$ with parameters $(m,\sigma)$;
\item $\Phi_{(m,\sigma)}$ is a distribution function of  $\xi_{(m,\sigma)}$;
\item $\gamma_{(m,\sigma)}$ is a linear Gaussian measure on $R$ defined by $\Phi_{(m,\sigma)}$.
\end{enumerate}

For $n \in \mathbb{N}$, let $\mu_n$ be a linear Gaussian measure  $\gamma_{(0,\sigma_n)}$ such that
\begin{align}
\frac{1}{\sqrt{2 \pi} \sigma_n}\int_{-1/ 2}^{-1/ 2}e^{-\frac{t^2}{2\sigma_n^2}}dt\le \frac{1}{2^n}.
\end{align}
Such a measure always exists. Indeed, we can take under  $\mu_n$ such  a linear Gaussian measure  $\gamma_{(0,\sigma_n)}$ for which
$\sigma_n>\frac{2^n}{\sqrt{2\pi}}$ for each $n \in \mathbb{N}$.

Let us show that the product-measure $\prod_{n \in \mathbb{N}}\mu_n$ is a transverse to $\limsup_{n\to\infty} E_n$.

 We have to show that
\begin{align}
(\prod_{n \in \mathbb{N}}\mu_n)(\limsup_{n\to\infty} E_n +(h_n)_{n \in \mathbb{N}})=0
\end{align}
for each $(h_n)_{n \in \mathbb{N}} \in  R^\mathbb{N}$.
Note that
\begin{align}
\limsup_{n\to\infty}E_n+(h_n)_{n \in \mathbb{N}}=\limsup_{n\to\infty}E^{(h_n)}_n,
\end{align}
where
\begin{align}
E^{(h_n)}_n=\{ (x_k)_{k \in \mathbb{N}}: x_n \in [-\frac{1}{2}+h_n, \frac{1}{2}+h_n]~\& ~x_k \in \mathbf{R}~\mbox{for}~ k \in \mathbb{N} \setminus \{n\} \}.\end{align}

Note that
\begin{align}
(\prod_{n \in \mathbb{N}}\mu_n)(E^{(h_n)}_n) = \mu_n ([-\frac{1}{2}+h_n, \frac{1}{2}+h_n]) \le \mu_n ([-\frac{1}{2}, \frac{1}{2}]) \le \frac{1}{2^n}.
\end{align}
The latter relation guaranties that
\begin{align}
\sum_{n=1}^{\infty}(\prod_{n \in \mathbb{N}}\mu_n)(E^{(h_n)}_n)<\infty,
\end{align}
which according to  Borel-Canteli Lemma  implies that
\begin{align}
(\prod_{n \in \mathbb{N}} \mu_n)\limsup_{n \to \infty} E^{(h_n)}_n)=0.
\end{align}

Since $E_n$ is Borel measurable in $R^{N}$ for each $n \in N$, we deduce that $\limsup_{n\to\infty} E_n$ also is Borel measurable.
Now we claim that $\limsup_{n\to\infty} E_n$ is a Borel shy set because $\prod_{n \in \mathbb{N}}\mu_n$ is the measure transverse to $\limsup_{n\to\infty} E_n$. Finally, a set $\cup_{J \in \mathbb{J}} R^J \times [-1/ 2,1/ 2]^{\mathbb{N} \setminus J}$, as a subset of a  Borel shy set in $R^{N}$, by the Definition 2.3 is  shy  in  $R^{N}$.

This ends the proof of the lemma.

\end{proof}
The next proposition  is a simple consequence of Lemma 3.2 and the inclusion (3.3).

\begin{thm} The set $D$ of all real valued sequences uniformly distributed  in $[-1/ 2,1/ 2]$ is shy  in $R^N$.

\end{thm}

\section{On structure of real valued sequences  uniformly distributed modulo 1 in $[-1/ 2,1/ 2]$
from the point of view of shyness}

\begin{defin}
A sequence of real
numbers $(x_k)_{k \in \mathbb{N}} \in \mathbf{R}^{\infty}$ is said to be uniformly distributed modulo 1 (abbreviated u.d. mod 1) if for each $c,d$ with $0 \le c
<d \le 1$  we have
\begin{align}
\lim_{n \to \infty}\frac{\#(\{ \{x_k\} : 1 \le k \le n\} \cap
[c,d])}{n}=d-c.
\end{align}
We denote by $E$ the set of all real valued sequences uniformly distributed modulo 1.

\end{defin}

\begin{defin}
A sequence of real
numbers $(x_k)_{k \in \mathbb{N}} \in \mathbf{R}^{\infty}$ is said to be uniformly distributed modulo 1 in  $[-1/2,1/2]$ (abbreviated u.d. mod 1 in $[-1/2,1/2]$) if for each $c,d$ with $-1/2 \le c
<d \le 1/2$  we have
\begin{align}
\lim_{n \to \infty}\frac{\#(\{ \{x_k\}-1/2 : 1 \le k \le n\} \cap
[c,d])}{n}=d-c.
\end{align}
We denote by $F$ the set of all real valued sequences uniformly distributed modulo 1.
\end{defin}

\begin{rem}Note that $(x_k)_{k \in \mathbb{N}}$ is uniformly distributed modulo 1 if and only if
$(x_k)_{k \in \mathbb{N}} $ is  uniformly distributed modulo 1 in  $[-1/2,1/2]$. Hence we have  that $E=F$.
\end{rem}

In the sequel we need the following lemma.

\begin{lem}(\cite{KuiNie74}, THEOREM 1.1, p. 2)  The sequence $(x_n)_{n \in N}$ of real numbers is
u.d. mod 1 if and only if for every real valued continuous function $f$ defined
on the closed unit interval $\overline{I}=[0, 1]$ we have
\begin{align}
\lim_{ N \to \infty}\frac{\sum_{n=1}^Nf(\{x_n\})}{N} =\int_{\overline{I}} f(x) dx.
\end{align}
\end{lem}

\begin{thm}(SM) The set $E$ of all real valued sequences  uniformly distributed modulo 1  is prevalent set in $\mathbf{R}^{\infty}$.

\end{thm}

\begin{proof} Let $E_0$ be the set of all sequences from $(0,1)^{\infty}$ which are not uniformly distributed in $[0,1]$. Since the measure $\lambda_{\infty}$ from Lemma 2.2 and the measure $\mu_{\mathbb{N}}$ from Lemma 2.9 coincides on subsets of $(0,1)^{\infty}$ in Solovay model, by   Lemma 2.2 we deduce that $\mu_{\mathbb{N}}(E_0)=0$.

By the definition of the functional $\mu_{\mathbb{N}}$ we have
\begin{align}
\mu_{\Bbb{N}}(E_0)= \int \limits_{\prod \limits_{k \in
\Bbb{N}}S_k}f_{E_0}(g)d \overline{\lambda_{\Bbb{N}}}(g))=0.
\end{align}

We put

\begin{align}
X_n=\{ g: g \in \prod \limits_{k \in \Bbb{N}}S_k ~\&~\mbox{card}(\prod\limits_{k \in \Bbb{N}}f_k)^{-1}(g) \cap E_0)=n\}
\end{align}
for $n \in \Bbb{N} \cup \{ +\infty\}$.
Then we get
\begin{align}
f_{E_0}(g)=\sum_{n \in \Bbb{N} \cup \{ +\infty\}}n \chi_{X_n}(g).
\end{align}

Since

\begin{align}
\mu_{\Bbb{N}}(E_0)=\sum_{n \in \Bbb{N} \cup \{ +\infty\}}n \overline{\lambda_{\Bbb{N}}}(X_n)=0
\end{align}
and

\begin{align}
\mbox{card}( (\prod\limits_{k \in \Bbb{N}}f_k)^{-1}(g) \cap E_0) \le 1
\end{align}
for each
$g \in \prod \limits_{k \in
\Bbb{N}}S_k$, we claim that
\begin{align}
\overline{\lambda_{\Bbb{N}}}(X_n)=0
\end{align}
for each $n \in (\Bbb{N} \setminus \{0\} ) \cup \{ +\infty\}$, which implies that
\begin{align}
\overline{\lambda_{\Bbb{N}}}(X_0)=1.
\end{align}

Now let $E^{*}$ be set of all sequences of real numbers which are not uniformly distributed modulo 1. Then
we get

\begin{align}
f_{E^{*}}(g)=\sum_{n \in \Bbb{N} \cup \{ +\infty\}}n \chi_{Y_n}(g)
\end{align}
where
\begin{align}
Y_n=\{ g: g \in \prod \limits_{k \in
\Bbb{N}}S_k ~\&~\mbox{card}((\prod\limits_{k \in \Bbb{N}}f_k)^{-1}(g) \cap E^{*})=n\}.
\end{align}
Let us show that $X_0 \subseteq Y_0$. Assume the contrary.  Then for some $g \in X_0$  and $n >0$ we get
\begin{align}
0=\mbox{card}((\prod\limits_{k \in \Bbb{N}}f_k)^{-1}(g) \cap E_0) < \mbox{card}((\prod\limits_{k \in \Bbb{N}}f_k)^{-1}(g) \cap E^*)=n,
\end{align}
which implies an existence of such a sequence $(x_k)_{k \in \Bbb{N}} \in (\prod\limits_{k \in \Bbb{N}}f_k)^{-1}(g) \cap E^*$ for which
\begin{align}
(\{x_k\})_{k \in \Bbb{N}} \in (\prod\limits_{k \in \Bbb{N}}f_k)^{-1}(g) \cap E^*.\end{align}
Then we get also that
\begin{align}
(\{x_k\})_{k \in \Bbb{N}} \in (\prod\limits_{k \in \Bbb{N}}f_k)^{-1}(g) \cap E_0\end{align}
which is the contradiction and we proved that $X_0 \subseteq Y_0$.

Since $X_0 \subseteq Y_0$ and $\overline{\lambda_{\Bbb{N}}}(X_0)=1$, we claim that $\overline{\lambda_{\Bbb{N}}}(Y_0)=1.$
The latter relation implies that $\overline{\lambda_{\Bbb{N}}}(Y_n)=0$ for each $n \in (\Bbb{N} \setminus \{0\} ) \cup \{ +\infty\}$.
Finally we get
\begin{align}
\mu_{\Bbb{N}}(E^{*})=\sum_{n \in \Bbb{N} \cup \{ +\infty\}}n \overline{\lambda_{\Bbb{N}}}(Y_n)=0.
\end{align}
Since $\mu_{\Bbb{N}}$ is the completion of a quasi-finite translation-invariant Borel measure in $R^{\Bbb{N}}$, by Lemma 2.6 we easily deduce that
  $\mu_{\Bbb{N}}$ is the generator of shy sets in $R^{\Bbb{N}}$ which implies that $E^{*}$ is shy. The latter relation implies that the set  $R^{\Bbb{N}} \setminus E^{*}$, being the set $E$  of all real valued sequences uniformly distributed modulo $1$, is prevalent set in $\mathbf{R}^{\infty}$.

This ends the proof of the theorem.

\end{proof}

By using Remark 4.4, we get the following corollary of Theorem 4.5.

\begin{cor}(SM) The set $F$ of all real valued sequences, uniformly distributed modulo 1 in  $[-1/2,1/2]$,   is prevalent set in $\mathbf{R}^{\infty}$.

\end{cor}

By using Lemma 4.4 and Theorem 4.5 we get the following versions of the strong law of large numbers in terms of prevalent set.

\begin{cor}(SM)  Let $f$ be a real valued continuous function  defined
on the closed unit interval $\overline{I}=[0, 1]$. Then
\begin{align}
\{ (x_n)_{n \in \mathbb{N}} \in  R^{\mathbb{N}} ~:~ \lim_{ N \to \infty}\frac{\sum_{n=1}^Nf(\{x_n\})}{N} =\int_{\overline{I}} f(x) dx \}
\end{align}
is prevalent set in $R^{\mathbb{N}}$.
\end{cor}

\begin{cor}(SM)  The set
\begin{align}
\cap _{ f \in  C[0,1]} \{ (x_n)_{n \in \mathbb{N}} \in  R^{\mathbb{N}} ~:~ \lim_{ N \to \infty}\frac{\sum_{n=1}^Nf(\{x_n\})}{N} =\int_{\overline{I}} f(x) dx \}
\end{align}
is prevalent set in $R^{\mathbb{N}}$.
\end{cor}
\begin{proof} By Lemma 4.4, we know  that
\begin{align}
E \subseteq \{ (x_n)_{n \in \mathbb{N}} \in  R^{\mathbb{N}} ~:~ \lim_{ N \to \infty}\frac{\sum_{n=1}^Nf(\{x_n\})}{N} =\int_{\overline{I}} f(x) dx \}
\end{align}
for each $f \in  C[0,1]$.  The latter relation implies that 
\begin{align}
E \subseteq \cap _{ f \in  C[0,1]} \{ (x_n)_{n \in \mathbb{N}} \in  R^{\mathbb{N}} ~:~ \lim_{ N \to \infty}\frac{\sum_{n=1}^Nf(\{x_n\})}{N} =\int_{\overline{I}} f(x) dx \}.
\end{align}
Application of the result of Theorem 4.5  ends the proof of the corollary.
\end{proof}

\medskip
\noindent {\bf Acknowledgment}. The authors wish to thank the referees for their constructive critique of the first draft.
 The research for this paper  was partially supported by Shota
Rustaveli National Science Foundation's Grant no 31/25.

\end{document}